\documentclass[leqno,11pt]{amsart}

\usepackage{amsmath}
\usepackage{graphicx, color}
\usepackage{amscd}
\usepackage{amsfonts}
\usepackage{amssymb}
\usepackage{mathrsfs}
\usepackage{mathtools}

\newtheorem{thm}{Theorem}[section]

\newtheorem{lem}[thm]{Lemma}

\newtheorem{rem}{Remark}[section]

\numberwithin{equation}{section}

\newcommand\be{\begin{equation}}
\newcommand\ee{\end{equation}}
\newcommand\R{\mathbb R}

\allowdisplaybreaks
\def\eps{\varepsilon}

\title[Liouville-type theorem]
{A Liouville-type theorem for an elliptic equation with superquadratic growth in the gradient}

\author[Roberta Filippucci]{Roberta Filippucci \textsuperscript{1}}

\author[Patrizia Pucci]{Patrizia Pucci \textsuperscript{1}}

\author[Philippe Souplet]{Philippe Souplet \textsuperscript{2}}

\thanks{\textsuperscript{1}
Dipartimento di Matematica e Informatica, Universit\'a degli Studi di Perugia
Via Vanvitelli 1, I-06123 Perugia, Italy.
E-mails: roberta.filippucci@unipg.it; patrizia.pucci@unipg.it}

\thanks{\textsuperscript{2}
Universit\'e Paris 13, Sorbonne Paris Cit\'e,
CNRS UMR 7539, Laboratoire Analyse, G\'{e}om\'{e}trie et Applications,
93430 Villetaneuse, France. 
E-mail: souplet@math.univ-paris13.fr}

\begin{document}

\begin{abstract}
We consider the elliptic equation $-\Delta u = u^q|\nabla u|^p$ in $\mathbb R^n$ for any $p\ge 2$ and $q>0$.
 We prove a Liouville-type theorem, which asserts that any positive bounded solution is constant.
 The proof technique is based on monotonicity properties for the spherical averages  of sub- and super-harmonic functions,
combined with a gradient bound obtained by a local Bernstein argument.
 This solves, in the case of bounded solutions, 
 a  problem left open in~\cite{BVGHV}, where the authors consider the case $0<p<2$.
Some extensions to elliptic systems are also given.
\end{abstract}

\maketitle

\section{Introduction and main results}

In this paper we are interested in proving a Liouville-type result for positive solutions of the elliptic equation
\be\label{main_pb}
-\Delta u = u^q|\nabla u|^p \quad\text{ in } \R^n,
\ee
where $p\ge2$, $q>0$ and $n\ge 1$. 
Equation~\eqref{main_pb}  when $p=0$  reduces to the celebrated Lane-Emden equation and in this case the 
well-known and deep 
result of %
Gidas and Spruck in~\cite{GS81} asserts that if 
 $q<q_S$,
where $q_S=(n+2)/(n-2)_+$ is the critical Sobolev exponent,
 then no positive solutions can exist.   The result is sharp  since it fails for $q\ge q_S$. 
 In particular, for $q=q_S$ with $n\ge 3$,
 equation~\eqref{main_pb} admits the positive bounded Lane-Emden solutions 
  $$u(x)=\biggl(\frac{c\alpha}{\alpha^2+|x|^2}\biggr)^{(n-2)/2},\quad \alpha>0,\quad c=\sqrt{n(n-2)}.$$
 If we now consider supersolutions of the Lane-Emden equation, 
 the Liouville result remains true in the smaller range $q\le q_*$,
 where $q_*=n/(n-2)_+$ is the so-called Serrin critical exponent (cf.~e.g.~\cite{szActa}), and this condition is optimal for supersolutions; see \cite{ASCPDE} for a detailed description of some
 results in this direction.
 Of course, due to the large number of papers dealing with this topic and its generalizations to
 problems involving quasilinear elliptic operators, it is not possible to produce here an exhaustive bibliography.

 The subcase of~\eqref{main_pb}  when $q=0$ reduces to the well-known
 diffusive Hamilton-Jacobi equation.  It was studied in~\cite{Lions85} and it was proved there that any classical solution
  has to be constant if $p>1$. Thus, in that case, nonexistence holds without  sign condition.

 Our main result is the following.

\begin{thm}
\label{mainthm}
Let $u$ be a positive  bounded classical solution of~\eqref{main_pb},
with $p\ge 2$  and $q>0$. Then $u$ is constant.
\end{thm}

\begin{rem}{\rm
(a) The result of Theorem~\ref{mainthm} is delicate since the conclusion fails for supersolutions.
Namely, for $p,q\ge 0$, there exists a positive, nonconstant bounded classical solution of
\be\label{main_pb2}
-\Delta u \ge u^q|\nabla u|^p \quad\text{ in } \R^n
\ee
 whenever $n\ge 3$ and
\begin{equation}
\label{cond_ex}
(n-2)q+(n-1)p>n.
\end{equation} 
Such supersolution can be found under the form $c(1+|x|^2)^{-\beta}$
for suitable $\beta,c>0$.
 Condition~\eqref{cond_ex} is essentially optimal, at least in the superlinear range. 
Indeed, if
$$(n-2)q+(n-1)p\le n \quad\hbox{and}  \quad p>1,$$
then any positive solution of \eqref{main_pb2} must be constant;
see \cite[Theorem~7.1]{cm} and \cite[Theorem~15.1]{mp}
 (and this remains of course true when $n\le 2$ %
since any positive superharmonic function is then constant).
We also refer to \cite{mp}, \cite{Fil09} for extensions of this result to quasilinear problems.

\smallskip

(b)  Although the  case of equation~\eqref{main_pb} with negative $q$ does not seem to have been 
much studied in the literature,
it is worth pointing out that Theorem~\ref{mainthm} remains true for all  $q>1-p$ (with $p\ge 2$),
as can be checked by inspection of the proof (see also Remark~\ref{rempq}). 

\smallskip

(c) We point out that, in the case $q=0$, Theorem~\ref{mainthm}, treated in \cite{Lions85} and already discussed above,   holds without assuming that $u$ is bounded.

Equation \eqref{main_pb} with $q\in (0,2)$, $p+q>1$, was studied in
detail in~\cite{BVGHV}, cf.~in particular %
Corollary B-1,
and various regions for nonexistence were determined.
 We also refer to the earlier paper~\cite{BPGMQ} where the case $q\in (0,N/(N-1))$ was considered,
 cf.~\cite[Corollary~2]{BPGMQ}.
The case $q\ge 2$ was left essentially open, and we completely answer it here  
 in the case of bounded solutions.
Note that Theorem~\ref{mainthm} is true without boundedness assumption in the radial case,
as shown in~\cite{BVGHV}.
We do not know if the conclusion of~Theorem~\ref{mainthm} remains true without
the boundedness assumption on $u$ in the nonradial case for $n\ge 3$.

Let us briefly consider the coercive analogue of \eqref{main_pb},
namely $\Delta u=u^q|\nabla u|^p$.
This problem  is quite different since, in the case $p=0$ (the so-called Keller-Osserman problem) the threshold value of $q$ for nonexistence does not depend on the dimension. 
  Early results for $p>0$ appeared in~\cite{MarPor}, while more recent ones can be found in~\cite{fprgrad}, \cite{farserrII} and the references therein. 
 See also~\cite{Ghergu} for results on radial solutions of related coercive systems.
}
\end{rem}

 The basic idea of our proof is to use the monotone nonincreasing property of the spherical averages of 
the superharmonic function $u$ (cf.~Lemma \ref{lemharm}).
A key observation is then that a suitably large power of the function $v:=u-\inf u$ is 
on the contrary {\it subharmonic}.
The latter property is a consequence of a suitable gradient bound (cf.~Lemma \ref{lembound}),
obtained by a local Bernstein argument.
The combination of these two opposite monotonicity properties eventually forces $u$ to be constant.

We stress that this proof is quite different from those in~\cite{Lions85} and~\cite{BVGHV}. Indeed, the celebrated result
by Lions for $q=0$, cf.~Corollary~IV of~\cite{Lions85}, is a  direct consequence  of a local Bernstein-type estimate
 of the gradient of generic $C^2$ solutions of $-\Delta u=|\nabla u|^p$ in a ball $B_R$, 
 an estimate which forces the vanishing of the gradient upon letting $R\to\infty$.
On the other hand, the Liouville Theorem~B in~\cite{BVGHV} in the range $0<p<2$ 
(plus additional assumptions on $p,q$)
 is also proved by Bernstein-type arguments (significantly more delicate and technical than in the case $q=0$.)

\medskip

Our methods can be used for more general problems, including systems.
As one of the possible extensions, let us consider the following class of systems:
\be\label{extsyst}
\begin{cases}
-\Delta u=f(x,u,v,\nabla u,\nabla v)&\mbox{ in } \R^n,
\vspace{1mm}\\
-\Delta v=g(x,u,v,\nabla u,\nabla v) &\mbox{ in } \R^n,
\end{cases}
\ee
where $f, g: \R^n\times[0,\infty)^2\times\R^n\times\R^n\to \R$.
We have the following result.

\begin{thm}
\label{mainthm2}
Let $f,g$ be nonnegative.
\smallskip

(i) Assume that, for each $M>0$, there exists a constant $K=K(M)>0$ such that,
\be\label{hypmainthm2}
\begin{aligned}
&\hbox{for all $(u,v,\xi,\zeta)\in \Gamma_M:= [0,M]^2\times\R^n\times\R^n$,} \\
&\quad uf(x,u,v,\xi,\zeta)\le K(M)|\xi|^2 \quad\hbox{ and }\quad vg(x,u,v,\xi,\zeta)\le C(M)|\zeta|^2.
\end{aligned}
\ee
Then any nonnegative, bounded, classical solution $(u,v)$ of~\eqref{extsyst} 
must be constant.
\smallskip

(ii) Assume that property \eqref{hypmainthm2} holds with
 $\Gamma_M$ replaced by $[0,M]^2\times B_M\times B_M$.
Let $(u,v)$ be a nonnegative classical solution of~\eqref{extsyst} 
such that $u,v,|\nabla u|,|\nabla v|$ are bounded.
Then $(u,v)$ must be constant.
\end{thm}

\begin{rem}{\rm
(a) Theorem~\ref{mainthm2}(i) applies in particular for the model system
\be\label{extsyst2} 
\begin{cases}
-\Delta u=u^{q_1}v^{r_1}|\nabla u|^{p_1}&\mbox{ in } \R^n,
\vspace{1mm}\\
-\Delta v=v^{q_2}u^{r_2}|\nabla v|^{p_2}&\mbox{ in } \R^n,
\end{cases}
\ee
with $p_1=p_2=2$ and any $q_1,q_2,r_1,r_2\ge 0$, whereas Theorem~\ref{mainthm2}(ii) applies for any $p_1,p_2\ge 2$
and $q_1,q_2,r_1,r_2\ge 0$.
 We refer to e.g. \cite{Fil11}, \cite{Fil13} for related results on systems of inequalities 
 corresponding to \eqref{extsyst2}, 
under stronger restrictions on the exponents.

\smallskip

(b) We note that, in the scalar case (cf.~Theorem~\ref{mainthm}), no boundedness assumption on the gradients was necessary.
This is due to the possibility of proving a suitable gradient estimate by a Bernstein-type argument (cf.~Lemma \ref{lembound}),
a property which does not seem available in general for systems.
\smallskip

(c) As can be seen from the proof, Theorem~\ref{mainthm2} extends in a straightforward manner to systems of more than 
two equations.
}
\end{rem}

\section{Proofs}

In view of the proof of Theorem~\ref{mainthm} we prepare two lemmas.
Our first lemma is a gradient bound, that will be obtained by a local Bernstein argument.
\begin{lem}
\label{lembound}
Let $u$ be a positive  bounded classical solution of~\eqref{main_pb},
with $p> 2$ and $q>0$. Then $u^{q+1}|\nabla u|^{p-2}$ is bounded.
\end{lem}

Next, for any function $w$, we denote by
$$\bar w(R)=|S_R|^{-1}\int_{S_R}w \, d\sigma,\quad R>0,$$
the spherical average of $w$.
We shall use the  next result (see e.g.~\cite[Lemma 3.2]{QS}).

\begin{lem}\label{lemharm}
If $w\ge 0$ is superharmonic, then $\bar w$ is nonincreasing in $\mathbb R^+$ and
$$w(x)\ge \lim_{R\to\infty} \bar w(R)\in [0,\infty),\quad\hbox{ for all $x\in\R^n$.}$$
\end{lem}

\begin{proof}[Proof of Lemma \ref{lembound}]
Set
$$u=v^m$$
 for $m>0$ to be chosen later. We compute
$$\nabla u=mv^{m-1}\nabla v,\qquad \Delta u=mv^{m-1}\Delta v+m(m-1)v^{m-2}|\nabla v|^2.$$
Substituting in \eqref{main_pb}, we obtain
$$-mv^{m-1}\Delta v-m(m-1)v^{m-2}|\nabla v|^2 = m^{p}v^{mq+(m-1){p}}|\nabla v|^{p}.$$
Hence, dividing by $mv^{m-1}$, we get
\begin{equation}\label{eq_v}
-\Delta v=(m-1)\frac{|\nabla v|^2} {v}+m^{p-1}v^{mq+(m-1)(p-1)}|\nabla v|^p.
\end{equation}
For given $\alpha\in (0,1)$ any $x_0\in \mathbb R^n$, let $\eta\in C^2(\mathbb R^n)$ be a cut-off function such that $0\le \eta\le 1$ in $\mathbb R^n$, $\eta=1$ in
$B_1(x_0)$ and $\eta=0$ in $B_2^C(x_0)$ and such that
\begin{equation}\label{cutoff_properties}
|\nabla \eta|\le C\eta^\alpha, \quad |\Delta \eta|+\frac{|\nabla \eta|^2}{\eta}\le C\eta^\alpha\quad\mbox{in}\ B_{2},
\end{equation}
with $C=C(\alpha,n)>0$ independent of $x_0$.
Put
\begin{equation}\label{def_z}
w=|\nabla v|^2,\qquad z:= \eta w, 
\end{equation}
so that
$$\Delta z= \mbox{div}\bigl(|\nabla v|^2 \nabla \eta+\eta \nabla (|\nabla v|^2)\bigr)
= |\nabla v|^2\Delta \eta +2 \langle \nabla (|\nabla v|^2), \nabla \eta\rangle +\eta \Delta (|\nabla v|^2).$$
 Recalling the Bochner formula $\Delta w=2\langle \nabla (\Delta v), \nabla v\rangle +2 |D^2v|^2$,
where $|D^2v|^2=\sum_{ij}(v_{ij})^2$, and using \eqref{eq_v}, we compute
$$\begin{aligned}
\langle \nabla (\Delta v), \nabla v\rangle&=\left\langle \nabla \biggl(-\frac {m-1}vw -m^{p-1}v^{mq+(m-1)(p-1)}w^{p/2}\biggr),\nabla v\right\rangle\\
&=\biggl[-\frac {m-1}v-\frac p2m^{p-1}v^{mq+(m-1)(p-1)}w^{(p-2)/2}\biggr]\langle \nabla w, \nabla v\rangle \\
&\quad-m^{p-1}[mq+(m-1)(p-1)]v^{m(q+p-1)-p}w^{(p+2)/2}+\frac {m-1}{v^2}w^2.
\end{aligned}$$
Therefore, $w$ is a solution of
\begin{gather*}-\Delta w -\langle b, \nabla w\rangle +2 |D^2v|^2=2aw^2,\\
a=m^{p-1}[mq+(m-1)(p-1)]v^{m(q+p-1)-p}w^{(p-2)/2}-\frac {m-1}{v^2},\\
b= \biggl[2\frac {m-1}v+pm^{p-1}v^{mq+(m-1)(p-1)}w^{(p-2)/2}\biggr]\nabla v,
\end{gather*}
while $z$ is a solution of
$$
-\Delta z+2\eta |D^2v|^2
= -w\Delta \eta -2\langle \nabla w, \nabla \eta\rangle+\eta\langle b, \nabla w\rangle
+2\eta a w^2.$$ 
Using $\eta \nabla w=\nabla z-w\nabla \eta$, we obtain
\begin{align*}
-\Delta z -\langle b, \nabla z\rangle +2 \eta |D^2v|^2=-w[\Delta \eta+\langle b, \nabla \eta\rangle] -2\langle \nabla w, \nabla \eta\rangle
+2\eta a w^2.
 \end{align*}
As a consequence of the Cauchy-Schwarz inequality, we have
$$|\nabla w|= |\nabla (|\nabla v|^2)|=2|(D^2v)\nabla v|\le 2|D^2v||\nabla v|.$$
Thus,
$$|\langle \nabla w, \nabla \eta\rangle|\le 2|D^2v||\nabla v||\nabla \eta|\le 2\frac{|\nabla \eta|^2|\nabla v|^2}\eta  +\frac12 \eta |D^2v|^2=2
\frac{|\nabla \eta|^2}\eta w +\frac12 \eta |D^2v|^2.$$
Putting $\mathscr{L}z=-\Delta z -\langle b, \nabla z\rangle$, we get
\be\label{eqLz}
\mathscr{L}z + \eta|D^2v|^2\le \biggl[\mathcal L\eta + 4\frac{|\nabla \eta|^2}\eta\biggr] w +2\eta aw^2.
\ee
We now make the choice
\be\label{choicem}
mq+(m-1)(p-1)=-1 \Longleftrightarrow m(q+p-1)=p-2
\Longleftrightarrow m=\frac{p-2}{q+p-1}\in (0,1).
\ee
This yields
$$a=\Bigl[-m^{p-1}w^{(p-2)/2}+(1-m)\Bigr] v^{-2}$$
and
$$ b= \Bigl[pm^{p-1}w^{(p-2)/2}-2(1-m)\Bigr]v^{-1}\nabla v.$$
Hence, owing to $p>2$,
\be\label{aw2}
\begin{aligned}
2\eta aw^2
&=2\eta\Bigl[-m^{p-1}|\nabla v|^{p+2}+(1-m)|\nabla v|^4\Bigr] v^{-2} \\
&\le \eta\Bigl[-\frac{3}{2}m^{p-1}|\nabla v|^{p+2}+C\Bigr] v^{-2}.
\end{aligned}
\ee
Here and in the rest of the proof, $C$ denotes a generic positive constant
depending only on $m,n,p,q$. 
For any $\eps>0$,
using
$$|b|w
\le C\Bigl(|\nabla v|^{p-2}+1\Bigr)v^{-1}|\nabla v|^3
\le C\Bigl(|\nabla v|^{p+1}+1\Bigr)v^{-1},$$
we get
\begin{align*}
\left(\mathcal L\eta + 4\frac{|\nabla \eta|^2}\eta\right) w
&\le \left(|\Delta \eta| + 4\frac{|\nabla \eta|^2}\eta\right)w+C|\nabla  \eta|\left(|\nabla v|^{p+1}+1\right)v^{-1} \\
&\le C\eta^\alpha\Bigl[|\nabla v|^2+(|\nabla v|^{p+1}+1)v^{-1} \Bigr],\\
&\le C\eta^\alpha\bigl(|\nabla v|^{p+1}\wedge 1\bigr)v^{-1},
 \end{align*}
  where we used \eqref{cutoff_properties}.
Taking $\alpha= (p+1)/(p+2)\in (1/2,1)$,
  using Young's inequality and the fact that $v$ is bounded, we have
\begin{align*}
\left(\mathcal L\eta +4\frac{|\nabla \eta|^2}\eta\right)w
&\le C\eta^\alpha\bigl(|\nabla v|^{p+1}\wedge 1\bigr)v^{-2\alpha}\|v\|_\infty^{2\alpha-1}\\
&\le \varepsilon\eta\bigl(|\nabla v|^{p+2}\wedge 1\bigr)v^{-2}+C_\varepsilon\|v\|_\infty^{\frac{2\alpha-1}{1-\alpha}}\\
&\le \varepsilon\eta(|\nabla v|^{p+2}+1)v^{-2}+C_\varepsilon\|v\|_\infty^p,
 \end{align*}
where $C_\eps$ depends only on $m,n,p,q,\eps$.
Combining this with~\eqref{eqLz} and~\eqref{aw2},
we obtain
$$\mathscr{L}z + \eta|D^2v|^2 \le  \varepsilon\eta(|\nabla v|^{p+2}+1)v^{-2}+C_\varepsilon\|v\|_\infty^p
+ \eta\Bigl[-\frac{3}{2}m^{p-1}|\nabla v|^{p+2}+C\Bigr] v^{-2}.$$
Now choosing $\eps=\frac12 m^{p-1}$, we get
$$\mathscr{L}z  \le \Bigl(-m^{p-1}\eta |\nabla v|^{p+2}+C\|v\|_\infty^{p+2}+C\Bigr) v^{-2}.$$
  Since $z\ge 0$ has compact support, it attains its maximum at some point
 $\xi=\xi(x_0)$, and at this point we have
$$
0\le \mathscr{L}z
 \le \Bigl[-m^{p-1}z^{(p+2)/2}+C\|v\|_\infty^{p+2}+C\Bigr] v^{-2}.
$$
Hence
  $$|\nabla v(x_0)|^2\le \|z\|_\infty=z(\xi)\le C[1+\|v\|_\infty]^2.$$
  Since $C$ is independent of $x_0$, we deduce that $|\nabla v|$ is bounded in $\mathbb R^n$, with
  $$\|\nabla v\|_\infty\le C[1+\|v\|_\infty].$$
From
  $$|\nabla v|^{p-2}=|\nabla u^\frac{p+q-1}{p-2}|^{p-2}=cu^{q+1}|\nabla u|^{p-2},$$
the conclusion of   the lemma follows at once.
\end{proof}

\begin{rem}\label{rempq}
{\rm
 Lemma \ref{lembound} remains valid for all $p>2$ and $q>1-p$.
Note that in that case we have $m>0$ instead of $m\in (0,1)$ in \eqref{choicem},
but the proof works without changes.}
\end{rem}

\begin{proof}[Proof of Theorem~\ref{mainthm}]
Since $u$ is superharmonic, Lemma~\ref{lemharm} gives
$$u(x)\ge \ell:=\lim_{R\to\infty} \bar u(R)\in [0,\infty),\quad\hbox{ for all $x\in\R^n$.}$$
Now, for $s\ge 2$ to be fixed later, set $z:=(u-\ell)^s\ge 0$.
By direct computation we have
\begin{align*}
\Delta z
&=s(u-\ell)^{s-1}\Delta u+s(s-1)(u-\ell)^{s-2}|\nabla u|^2\\
&=-s(u-\ell)^{s-1}u^q|\nabla u|^p+s(s-1)(u-\ell)^{s-2}|\nabla u|^2 \\
&=s(u-\ell)^{s-2}|\nabla u|^2\bigl[s-1-(u-\ell)u^q|\nabla u|^{p-2}\bigr] \\
&\ge s(u-\ell)^{s-2}|\nabla u|^2\bigl[s-1-u^{q+1}|\nabla u|^{p-2}\bigr].
\end{align*}
Now choose $s\ge 1+\|u^{q+1}|\nabla u|^{p-2}\|_\infty$,
noting that the latter quantity is finite thanks to Lemma~\ref{lembound} 
 (this lemma is actually not required if $p=2$). 
 It then follows that $\Delta z\ge 0$.

Thus, $z$ is subharmonic and bounded. Applying Lemma~\ref{lemharm} to
the superharmonic function $\|z\|_\infty-z\ge 0$, it follows  that $\bar z$ is nondecreasing and that
$$z(x)\le k:=\lim_{R\to\infty} \bar z(R),\quad x\in\R^n.$$
But
$$\bar z(R)\le \|u\|_\infty^{s-1}(\bar u(R)-\ell)\to 0,\quad\hbox{ as $R\to \infty$.}$$
We conclude that $z\equiv 0$, i.e. $u\equiv \ell$.
\end{proof}

\begin{proof}[Proof of Theorem~\ref{mainthm2}]
We shall only prove assertion (i) since the proof of assertion (ii) is completely similar.
Let $(u,v)$ be a classical solution of~\eqref{main_pb} such that $0\le u,v\le M$.
Since $u$ is superharmonic, Lemma~\ref{lemharm} gives
$$u(x)\ge \ell:=\lim_{R\to\infty} \bar u(R)\in [0,\infty),\quad\hbox{ for all $x\in\R^n$.}$$
Now, for $s\ge K(M)+1$ with $K(M)$ from \eqref{hypmainthm2}, we set $z:=(u-\ell)^s\ge 0$.
Again, direct computation shows that
\begin{align*}
\Delta z
&=s(u-\ell)^{s-1}\Delta u+s(s-1)(u-\ell)^{s-2}|\nabla u|^2\\
&=-s(u-\ell)^{s-1}f(x,u,v,\nabla u,\nabla v)+s(s-1)(u-\ell)^{s-2}|\nabla u|^2 \\
&\ge s(u-\ell)^{s-2}\bigl[K(M)|\nabla u|^2-uf(x,u,v,\nabla u,\nabla v)\bigr] \ge 0,
\end{align*}
where we used assumption \eqref{hypmainthm2}.
Thus, $z$ is subharmonic and bounded. Applying Lemma~\ref{lemharm} to
the superharmonic function $\|z\|_\infty-z\ge 0$, it follows  that $\bar z$ is nondecreasing and that
$$z(x)\le k:=\lim_{R\to\infty} \bar z(R),\quad x\in\R^n.$$
But
$$\bar z(R)\le \|u\|_\infty^{s-1}(\bar u(R)-\ell)\to 0,\quad\hbox{ as $R\to \infty$.}$$
It follows that $z\equiv 0$, i.e. $u\equiv \ell$.
By exchanging the roles of $u,v$ and of $f,g$, we deduce that $v$ is also constant.
This completes the proof.
\end{proof}

{\bf Acknowledgements.} Part of this work was done during a visit of
PhS at the
Dipartimento di Matematica e Informatica of the Universit\`a degli Studi di Perugia
 within the auspices of the INdAM -- GNAMPA Projects
2018. He wishes to thank this institution for the kind hospitality.
PhS is partly supported by the Labex MME-DII (ANR11-LBX-0023-01).
\smallskip

RF and PP were partly supported by the Italian MIUR project
{\em Variational methods, with applications to problems in mathematical physics and
geometry} (2015KB9WPT\_009) and are members of the {\em Gruppo Nazionale per
l'Analisi Ma\-te\-ma\-ti\-ca, la Probabilit\`a e le loro Applicazioni}
(GNAMPA) of the {\em Istituto Nazionale di Alta Matematica} (INdAM).
The manuscript was realized within the auspices of the INdAM -- GNAMPA Projects
2018
{\em Pro\-ble\-mi non lineari alle derivate parziali} (Prot\_U-UFMBAZ-2018-000384).

\end{document}